\documentclass[11pt]{article}
\usepackage{amsmath,amssymb,amsfonts}
\usepackage{mathrsfs,mathtools,bm}
\usepackage{graphicx}
\usepackage{tikz}
\usetikzlibrary{calc}
\usetikzlibrary{positioning}

\usepackage{xcolor}
\usepackage{geometry}
\usepackage[colorlinks=true,
linkcolor=blue,citecolor=blue,
urlcolor=blue]{hyperref}

\makeatletter
\def\@seccntDot{.}
\def\@seccntformat#1{\csname the#1\endcsname\@seccntDot\hskip 0.5em}
\renewcommand\section{\@startsection{section}{1}{\z@}%
{18\p@ \@plus 6\p@ \@minus 3\p@}%
{9\p@ \@plus 6\p@ \@minus 3\p@}%
{\large\bfseries\boldmath}}
\renewcommand\subsection{\@startsection{subsection}{2}{\z@}%
{12\p@ \@plus 6\p@ \@minus 3\p@}%
{3\p@ \@plus 6\p@ \@minus 3\p@}%
{\bfseries\boldmath}}
\renewcommand\subsubsection{\@startsection{subsubsection}{3}{\z@}%
{12\p@ \@plus 6\p@ \@minus 3\p@}%
{\p@}%
{\bfseries\boldmath}}
\makeatother

\usepackage{microtype}

\usepackage[amsmath,thmmarks]{ntheorem}
\theoremstyle{plain}
\newtheorem{theorem}{Theorem}[section]
\newtheorem{lemma}{Lemma}[section]
\newtheorem{corollary}{Corollary}[section]

\theorembodyfont{\normalfont}

\newtheorem{remark}{Remark}[section]

\newtheorem{claim}{Claim}[section]

\theoremstyle{nonumberplain}
\theoremseparator{}
\theoremsymbol{\ensuremath{\Box}}
\newtheorem{proof}{\it Proof.}

\numberwithin{equation}{section}
\allowdisplaybreaks
\begin{document}

\title{The extremal $p$-spectral radius of Berge-hypergraphs}
\author{
Liying Kang
\thanks{Department of Mathematics, Shanghai University, Shanghai 200444, PR China
({\tt lykang@i.shu.edu.cn})}
\and Lele Liu
\thanks{Department of Mathematics, Shanghai University, Shanghai 200444, PR China
({\tt ahhylau@gmail.com})}
\and Linyuan Lu
\thanks{Department of Mathematics, University of South Carolina, Columbia, SC 29208,
USA ({\tt lu@math.sc.edu})}
\and Zhiyu Wang
\thanks{Department of Mathematics, University of South Carolina, Columbia, SC 29208,
USA ({\tt zhiyuw@email.sc.edu})}}

\maketitle

\begin{abstract}
Let $G$ be a graph. We say that a hypergraph $H$ is a Berge-$G$ if there is a bijection
$\phi: E(G)\to E(H)$ such that $e\subseteq \phi(e)$ for all $e\in E(G)$. For any $r$-uniform
hypergraph $H$ and a real number $p\geq 1$, the $p$-spectral radius
$\lambda^{(p)}(H)$ of $H$ is defined as
\[
\lambda^{(p)}(H):=\max_{\bm{x}\in\mathbb{R}^n,\,||\bm{x}||_p=1}
r\sum_{\{i_1,i_2,\ldots,i_r\}\in E(H)} x_{i_1}x_{i_2}\cdots x_{i_r}.
\]
In this paper, we study the $p$-spectral radius of Berge-$G$ hypergraphs.
We determine the $3$-uniform hypergraphs with maximum $p$-spectral
radius for $p\geq 1$ among Berge-$G$ hypergraphs when $G$ is a path, a cycle or a star.
\par\vspace{2mm}

\noindent{\bfseries Keywords:} Uniform hypergraph; $p$-spectral radius; Berge-hypergraph.
\par\vspace{2mm}

\noindent{\bfseries AMS classification:} 05C65; 15A18.
\end{abstract}

\section{Introduction}
\label{sec1}

Let $G$ be a graph and $H$ be a hypergraph. The hypergraph $H$ is a {\em Berge-$G$} if there is
a bijection $\phi: E(G)\to E(H)$ such that $e\subseteq \phi(e)$ for all $e\in E(G)$. Alternatively,
$H$ is a Berge-$G$ if we can embed each edge of $G$ into each edge of $H$. Note that for a graph
$G$ there are in general many non-isomorphic hypergraphs which are Berge-$G$. An {\em $r$-uniform hypergraph}
or simply $r$-graph is a hypergraph such that all its edges have size $r$. We denote the family of
all $r$-uniform hypergraphs which are Berge-$G$ by $\mathcal{B}_r(G)$.

Recently, extremal problems for Berge-$G$ hypergraphs have attracted the attention of a lot of
researchers. Among those are extremal size of hypergraphs on $n$ vertices that contain no subhypergraph
from $\mathcal{B}_r(G)$, see for example \cite{Davoodi, Ergemlidze, FKL, FKL2, GMP, GerbnerPalmer2017, Gyori2006, GKL}. In addition, Ramsey
numbers for Berge-$G$ hypergraphs have been considered in
\cite{AxenovichGyarfas2018, GMOV, GyarfasSarkozy2011,SaliaTompkinsWangZamora2018}.

In this paper, we study the $p$-spectral radius of Berge-$G$ hypergraphs. The $p$-spectral radius
was introduced by Keevash, Lenz and Mubayi \cite{Keevash2014} and subsequently studied by Nikiforov
\cite{Nikiforov2014:analytic-methods,Nikiforov2014:extremal-problems,KangNikiforov2014}
and Chang et al. \cite{ChangDing2018}. Let $H$ be an $r$-uniform hypergraph of order $n$, the
polynomial form of $H$ is a multi-linear function $P_H(\bm{x}): \mathbb{R}^n\to\mathbb{R}$ defined
for any vector $\bm{x}=(x_1,x_2,\ldots,x_n)^{\mathrm{T}}\in\mathbb{R}^n$ as
\[
P_H(\bm{x})=r\sum_{\{i_1,i_2,\ldots,i_r\}\in E(G)}x_{i_1}x_{i_2}\cdots x_{i_r}.
\]
For any real number $p\geq 1$, the {\em $p$-spectral radius} of $H$ is defined as%
\footnote{We modified the definition of $p$-spectral radius by removing a constant
factor $(r-1)!$ from \cite{Keevash2014}, so that the $p$-spectral radius
is the same as the one in \cite{Cooper2012} when $p=r$. This is not essential
and does not affect the results at all.}
\begin{equation}\label{eq:definition-p-spectral-radius}
\lambda^{(p)}(H):=\max_{||\bm{x}||_p=1}P_H(\bm{x}),
\end{equation}
where $||\bm{x}||_p:=(|x_1|^p+|x_2|^p+\cdots+|x_n|^p)^{1/p}$.

For any real number $p\geq 1$, we denote $\mathbb{S}_{p,+}^{n-1}$ (resp. $\mathbb{S}_{p,++}^{n-1}$)
the set of all nonnegative (resp. positive) real vectors $\bm{x}\in\mathbb{R}^n$ with
$||\bm{x}||_p=1$. If $\bm{x}\in\mathbb{R}^n$ is a vector with $||\bm{x}||_p=1$ such that
$\lambda^{(p)}(H)=P_H(\bm{x})$, then $\bm{x}$ is called an {\em eigenvector} corresponding to $\lambda^{(p)}(H)$.
Note that $P_H(\bm{x})$ can always reach its maximum at some nonnegative vectors. By Lagrange's
method, we have the {\em eigenequations} for $\lambda^{(p)}(H)$ and $\bm{x}\in\mathbb{S}_{p,+}^{n-1}$
as follows:
\begin{equation}\label{eq:eigenequation}
\sum_{\{i,i_2,\ldots,i_r\}\in E(H)}x_{i_2}\cdots x_{i_r}=\lambda^{(p)}(H)x_i^{p-1}
~~\text{for}\ x_i>0.
\end{equation}

Note that the $p$-spectral radius $\lambda^{(p)}(H)$ shows remarkable connections with some hypergraph
invariants. For instance, $\lambda^{(1)}(H)/r$ is the Lagrangian of $H$ (see
\cite{Talbot2002,Lu2018:maximum-p-spectral-radius,GruslysLetzter2018}), $\lambda^{(r)}(H)$ is the usual
spectral radius introduced by Cooper and Dutle \cite{Cooper2012}, and $\lambda^{(\infty)}(H)/r$ is the
number of edges of $H$ (see \cite{Nikiforov2014:analytic-methods}).

To state our results precisely, we start with some basic definitions and notations. For graphs we denote
by $P_k$ the path on $k$ vertices, by $C_k$ the cycle on $k$ vertices and by $S_k$ the star with $k$
vertices. We also denote by $S_{k-1}^+$ the graph on $(k-1)$ vertices obtained from $S_{k-1}$ by adding
an edge (see Fig.\,\ref{fig:Sk^+}). For $r>s\geq 2$, let $R$ be a set of $(r-s)$ vertices and $H$ be a
$s$-uniform hypergraph. A {\em suspension} of $H$, denoted by $H\ast R$, is an $r$-uniform hypergraph
with vertex set $V(H)\cup R$ and edge set $\{e\cup R: e\in E(H)\}$. We call $H$ the {\em core} of $H\ast R$.

Let $\Delta_1$ be the graph on $(k-1)$ vertices obtained from a triangle by attaching two pendent paths
with lengths differing by at most one at a vertex of the triangle (see Fig.\,\ref{fig:Delta1}). Also,
we let $\Delta_2$ be the graph on $(k-1)$ vertices obtained from a triangle and a $(k-3)$-cycle by
identifying a vertex in $C_3$ and a vertex in $C_{k-3}$ (see Fig.\,\ref{fig:Delta2}). We can now
state our main results.

\begin{figure}[htbp]
\begin{minipage}{0.45\textwidth}
  \centering
  \begin{tikzpicture}[scale=0.9]
    \foreach \i in {-3,-2,...,3}
    \filldraw (\i,0) circle (0.06);
    \filldraw (240:1) circle (0.06);
    \filldraw (300:1) circle (0.06);
    \draw (-3,0)--(-2,0);
    \draw (-1,0)--(0,0)--(1,0);
    \draw (2,0)--(3,0);
    \draw (0,0)--(240:1)--(300:1)--(0,0);
    \node at (-1.5,0) {$\cdots$};
    \node at (1.5,0) {$\cdots$};
    \node at (-2,0.6) {$\overbrace{\hspace{18mm}}^{\left\lfloor\frac{k-4}{2}\right\rfloor}$};
    \node at (2,0.6) {$\overbrace{\hspace{18mm}}^{\left\lceil\frac{k-4}{2}\right\rceil}$};
  \end{tikzpicture}\vspace{2.9mm}
  \caption{\label{fig:Delta1}Graph $\Delta_{1}$}
\end{minipage}
\begin{minipage}{0.23\textwidth}
  \centering
  \begin{tikzpicture}[scale=0.9]
    \foreach \i in {-5,-4,...,1}
    {
    \coordinate (u\i) at (10+50*\i:0.9);
    \filldraw (u\i) circle (0.06);
    }
    \draw (u-5)--(u-4)--(u-3)--(u-2)--(u-1)--(u0)--(u1);
    \filldraw (80:0.9) circle (0.015);
    \filldraw (90:0.9) circle (0.015);
    \filldraw (100:0.9) circle (0.015);
    \filldraw (-0.5,-1.6) circle (0.06);
    \filldraw (0.5,-1.6) circle (0.06);
    \draw (u-2)--(-0.5,-1.6)--(0.5,-1.6)--(u-2);
  \end{tikzpicture}\vspace{-2mm}
  \caption{\label{fig:Delta2}Graph $\Delta_2$}
\end{minipage}
\begin{minipage}{0.3\textwidth}
  \centering
  \begin{tikzpicture}[scale=0.9]
    \foreach \i in {0,1,3,4}
    {
    \coordinate (u\i) at (30+30*\i:1.4);
    \filldraw (u\i) circle (0.06);
    \draw (0,0)--(u\i);
    }
    \filldraw (0,0) circle (0.06);
    \filldraw (-60:1) circle (0.06);
    \filldraw (-120:1) circle (0.06);
    \draw (0,0)--(-60:1)--(-120:1)--(0,0);
    \node at ($(u1)!0.5!(u3)$) {$\cdots$};
\end{tikzpicture}
\caption{\label{fig:Sk^+}Graph $S_{k-1}^+$}
\end{minipage}
\end{figure}

\begin{theorem}\label{thm:main-result-path}
Let $p\geq 1$, $k\geq 6$ and $H$ be a Berge-$P_k$ $3$-graph. Then
$\lambda^{(p)}(H)\leq\lambda^{(p)}(\Delta_1\ast K_1)$. Furthermore,
if $p>2$, the equality holds if and only if $H\cong \Delta_1\ast K_1$.
\end{theorem}

\begin{theorem}\label{thm:main-result-cycle}
Let $p\geq 1$, $k\geq 6$ and $H$ be a Berge-$C_k$ $3$-graph. Then
$\lambda^{(p)}(H)\leq\lambda^{(p)}(\Delta_2\ast K_1)$. Furthermore,
if $p>2$, the equality holds if and only if $H\cong \Delta_2\ast K_1$.
\end{theorem}

\begin{theorem}\label{thm:main-result-star}
Let $p\geq 1$ and $H$ be a Berge-$S_k$ $3$-graph.
\begin{enumerate}
  \item[$(1)$] If $p=1$, then $\lambda^{(1)}(H)\leq\lambda^{(1)}(S_{k-1}^+\ast K_1)=4/27$.

  \item[$(2)$] If $1<p<2$, then $\lambda^{(p)}(H)<\lambda^{(p)}(S_k\ast K_1)$ unless $H=S_k\ast K_1$ for sufficiently large $k$.

  \item[$(3)$] If $p\geq 2$ and $k\geq 11$, then $\lambda^{(p)}(H)<\lambda^{(p)}(S_k\ast K_1)$ unless $H=S_k\ast K_1$.
\end{enumerate}
\end{theorem}

The paper is organized as follows. In Section \ref{sec2}, we study some edge-shifting operations for
connected graphs/hypergraphs, and consider the effect of them to increase or decrease the $p$-spectral
radius. In Section \ref{sec3}, we give the proofs of \autoref{thm:main-result-path} and \autoref{thm:main-result-cycle}.
In Section \ref{sec4}, we give the proof of \autoref{thm:main-result-star}.

\section{Perturbations of $p$-spectral radius under edge operations}
\label{sec2}

In this section, we prove several results for comparing the $p$-spectral radius by edge-shifting operations.
Before continuing, we need the following part of Perron--Frobenius theorem for uniform hypergraphs.

\begin{theorem}[\cite{Nikiforov2014:analytic-methods}]
\label{thm:positive-vector-p>r-1}
Let $r\geq 2$ and $p>r-1$. If $H$ is a connected $r$-uniform hypergraph on $n$ vertices and
$(x_1,x_2,\ldots,x_n)^{\mathrm{T}}\in\mathbb{S}_{p,+}^{n-1}$ is an eigenvector corresponding to $\lambda^{(p)}(H)$,
then $x_i>0$ for $i=1$, $2$, $\ldots$, $n$.
\end{theorem}

The following edge-shifting operation was introduced in \cite{LiShaoQi2016:extremal-spectral-radius}.
Let $H$ be an $r$-uniform hypergraph with $u$, $v\in V(H)$ and $e_i\in E(H)$ such that $u\notin e_i$,
$v\in e_i$ for $i=1,2,\ldots,s$. Let $e_i'=(e_i\backslash\{v\})\cup\{u\}$, $i=1,2,\ldots,s$, and $H'$
be the $r$-uniform hypergraph with $V(H')=V(H)$ and $E(H')=(E(H)\backslash \{e_1,\ldots,e_s\})\cup \{e_1',\ldots,e_s'\}$.
Then we say that $H'$ is obtained from $H$ by moving edges $e_1$, $e_2$, $\ldots$, $e_s$ from $v$ to $u$.

In our proofs, we frequently use the following theorem that generalizes a result of
\cite[Theorem 15]{LiShaoQi2016:extremal-spectral-radius}. The proof is similar to that of
\cite[Theorem 15]{LiShaoQi2016:extremal-spectral-radius}.

\begin{theorem}\label{thm:move-edges}
Let $r\geq 2$ and $p\geq 1$. Suppose that $H$ is a connected $r$-uniform hypergraph, $H'$ is the hypergraph
obtained from $H$ by moving edges $e_1$, $e_2$, $\ldots$, $e_s$ from $v$ to $u$ and contains no multiple edges.
Let $\bm{x}\in\mathbb{S}_{p,+}^{v(H)-1}$ be an eigenvector corresponding to $\lambda^{(p)}(H)$. If $x_u\geq x_v$, then
$\lambda^{(p)}(H')\geq\lambda^{(p)}(H)$.
Furthermore, if $p>r-1$, then $\lambda^{(p)}(H')>\lambda^{(p)}(H)$.
\end{theorem}

\begin{proof}
Assume $p\geq 1$. In view of \eqref{eq:definition-p-spectral-radius},
we have
\begin{align*}
       & ~\lambda^{(p)}(H')-\lambda^{(p)}(H) \\
  \geq & ~r\sum_{\{i_1,\ldots,i_r\}\in E(H')} x_{i_1}\cdots x_{i_r}-
        r\sum_{\{i_1,\ldots,i_r\}\in E(H)} x_{i_1}\cdots x_{i_r} \\
  \geq & ~r(x_u-x_v)\sum_{i=1}^s \prod_{w\in e_i\backslash\{v\}} x_w \\
  \geq & ~0.
\end{align*}
Hence, $\lambda^{(p)}(H')\geq\lambda^{(p)}(H)$.

Now let $p>r-1$. By \autoref{thm:positive-vector-p>r-1}, $\bm{x}$ is a positive vector.
We shall prove $\lambda^{(p)}(H')\neq\lambda^{(p)}(H)$. If not, then $\bm{x}$ is also
the eigenvector of $H'$ corresponding to $\lambda^{(p)}(H')=\lambda^{(p)}(H)$. Using the
eigenequations \eqref{eq:eigenequation} for $u$, we have
\begin{align*}
  0 & =(\lambda^{(p)}(H')-\lambda^{(p)}(H))x_u^{p-1} \\
    & =\sum_{i=1}^s \prod_{w\in e_i\backslash\{v\}} x_w \\
    & >0,
\end{align*}
a contradiction. It follows that $\lambda^{(p)}(H')>\lambda^{(p)}(H)$. The proof is completed.
\end{proof}

For an $r$-uniform hypergraph $H$, the {\em link} $L(v)$ of a vertex $v\in V(H)$ is the $(r-1)$-uniform
hypergraph consisting of all $S\subseteq V(H)$ with $S\cup\{v\}\in E(H)$.

\begin{theorem}\label{thm:delete-vetex-increase-spectral-radius}
Let $r\geq 2$, $p\geq 1$, and $H$ be a connected $r$-uniform hypergraph. Suppose that $u$, $v\in V(H)$ such
that $u$, $v$ are not contained in an edge of $H$, and $L(u)\cap L(v)=\emptyset$. Let $H'$ be the hypergraph
obtained from $H$ by deleting $u$ and adding edges $\{f\cup\{v\}: f\in L(u)\}$. Then
$\lambda^{(p)}(H')\geq\lambda^{(p)}(H)$. Furthermore, if $p>r-1$, then $\lambda^{(p)}(H')>\lambda^{(p)}(H)$.
\end{theorem}

\begin{proof}
Let $p\geq 1$ and $\bm{x}\in\mathbb{S}^{v(H)-1}_{p,+}$ be an eigenvector corresponding to $\lambda^{(p)}(H)$.
We now define a vector $\bm{y}$ for $H'$ as follows:
\[
   y_w=
   \begin{cases}
     (x_v^p+x_u^p)^{1/p}, & \text{if}\ w=v,\\
     x_w, & \text{otherwise}.
   \end{cases}
\]
It is clear that $\bm{y}\in\mathbb{S}^{v(H')-1}_{p,++}$. Denote $R:=\{f\cup\{v\}: f\in L(u)\}$.
It follows from \eqref{eq:definition-p-spectral-radius} that
\begin{align*}
  \lambda^{(p)}(H') & \geq r\sum_{\{i_1,i_2,\ldots,i_r\}\in E(H')} y_{i_1}y_{i_2}\cdots y_{i_r} \\
                    & =r\sum_{e\in R} \prod_{w\in e} y_w+r\sum_{e\in E(H')\backslash R} \prod_{w\in e} y_w \\
                    & =r(y_v-x_u)\sum_{f\in L(u)} \prod_{w\in f}x_w+r\sum_{e\in E(H)} \prod_{w\in e}x_w \\
                    & \geq\lambda^{(p)}(H).
\end{align*}

If $p>r-1$, by \autoref{thm:positive-vector-p>r-1}, $\bm{x}\in\mathbb{S}^{v(H)-1}_{p,++}$ is a
positive eigenvector corresponding to $\lambda^{(p)}(H)$. Therefore, the above inequality is strict. This
completes the proof of the theorem.
\end{proof}

Suppose that $G$ is a connected graph and $u$ is a vertex in $G$. Let $G(u;k,s)$ be the graph obtained
from $G$ by attaching two pendent paths of lengths $k$ and $s$ at $u$, respectively.

\begin{figure}[htbp]
\begin{minipage}{0.45\textwidth}
  \centering
  \begin{tikzpicture}[scale=0.95]
     \draw[xshift=-1] (-0.77,0) node {$G$} circle (0.8);
     \filldraw (0,0) node[left] {$u$} circle (0.06);
     \foreach \i in {1,2,3,4}
     {
      \coordinate (u\i) at (-25:0.8*\i);
      \coordinate (v\i) at (25:0.8*\i);
      \filldraw (u\i) circle (0.06);
      \filldraw (v\i) circle (0.06);
     }
     \draw (0,0)--(u1)--(u2);
     \draw (0,0)--(v1)--(v2);
     \draw (u3)--(u4);
     \draw (v3)--(v4);
     \node[below=3pt] at (u1) {$u_1$};
     \node[below=3pt] at (u2) {$u_2$};
     \node[below=3pt] at (u4) {$u_{k+1}$};
     \node[above=3pt] at (v1) {$v_1$};
     \node[above=3pt] at (v2) {$v_2$};
     \node[above=3pt] at (v4) {$v_{s-1}$};
     \node[rotate=25] at (25:2.05) {$\cdots$};
     \node[rotate=-25] at (-25:2.05) {$\cdots$};
  \end{tikzpicture}
  \caption{\label{fig:G(k+1,s-1)} Graph $G(u;k+1,s-1)$}
\end{minipage}
\begin{minipage}{0.5\textwidth}
   \centering
   \begin{tikzpicture}[scale=0.95]
    \draw (0,-0.8) node {$G$} circle (0.8);
     \filldraw (0,0) node[below] {$u$} circle (0.06);
     \foreach \i in {1,2,...,5}
   {
      \coordinate (u\i) at (145:0.8*\i);
      \filldraw (u\i) circle (0.06);
   }
     \foreach \i in {1,2,...,5}
   {
     \coordinate (v\i) at (35:0.8*\i);
     \filldraw (v\i) circle (0.06);
   }
     \node[anchor=north east] at (u1) {$v_1$};
     \node[anchor=north east] at (u2) {$v_{s-i-1}$};
     \node[anchor=north east] at (u3) {$v_{s-i}$};
     \node[anchor=north east] at (u5) {$v_{s-1}$};
     \node[anchor=north west] at (v1) {$u_1$};
     \node[anchor=north west] at (v2) {$u_{k-i}$};
     \node[anchor=north west] at (v3) {$u_{k-(i-1)}$};
     \node[anchor=north west] at (v5) {$u_{k+1}$};
     \node[rotate=145] at ($(u1)!0.5!(u2)$) {$\cdots$};
     \node[rotate=145] at ($(u4)!0.5!(u5)$) {$\cdots$};
     \node[rotate=35] at ($(v2)!0.5!(v1)$) {$\cdots$};
     \node[rotate=35] at ($(v5)!0.5!(v4)$) {$\cdots$};
     \draw (v1)--(0,0)--(u1);
     \draw (u3)--(u4);
     \draw (v3)--(v4);
     \draw[dashed] (v3) edge[bend right=40] (u2);
     \draw[dashed] (v2) edge[bend right=40] (u3);
  \end{tikzpicture}
  \caption{\label{fig:G2} Graph $G_2$}
\end{minipage}
\end{figure}

\begin{theorem}\label{thm:G(k,s)>G(k+1,s-1)}
Let $p>1$ and $G$ be a connected graph. If $k\geq s\geq 1$, then
$\lambda^{(p)}(G(u;k,s))>\lambda^{(p)}(G(u;k+1,s-1))$.
\end{theorem}

\begin{proof}
Suppose for a contradiction that $\lambda^{(p)}(G(u;k,s))\leq\lambda^{(p)}(G(u;k+1,s-1))$.
By \autoref{thm:positive-vector-p>r-1}, we let $\bm{x}$ be a positive eigenvector of
$G(u;k+1,s-1)$ with $||\bm{x}||_p=1$ corresponding to $\lambda^{(p)}(G(u;k+1,s-1))$. Let
$uu_1\cdots u_ku_{k+1}$ and $uv_1\cdots v_{s-2}v_{s-1}$ be two pendent paths at $u$ of
lengths $(k+1)$ and $(s-1)$ in $G(u;k+1,s-1)$, respectively (see Fig.\,\ref{fig:G(k+1,s-1)}). For short, denote $v_0=u$.
We will prove by induction that
\[
x_{u_{k-i}}>x_{v_{s-i-1}},~i=0,1,\ldots,s-1.
\]

$\bullet$~~For $i=0$. If $x_{u_k}\leq x_{v_{s-1}}$, then for the graph $G_1$ obtained
by replacing edge $u_ku_{k+1}$ with $v_{s-1}u_{k+1}$, we have $G_1\cong G(u;k,s)$. By
\autoref{thm:move-edges}, $\lambda^{(p)}(G(u;k,s))=\lambda^{(p)}(G_1)>\lambda^{(p)}(G(u;k+1,s-1))$.
Hence, $x_{u_k}>x_{v_{s-1}}$.

$\bullet$~~Suppose that $x_{u_{k-(i-1)}}> x_{v_{s-(i-1)-1}}$. We shall show that
$x_{u_{k-i}}> x_{v_{s-i-1}}$. Let $G_2$ be the graph obtained from $G(u;k+1,s-1)$
by deleting edges $\{u_{k-i}u_{k-(i-1)},v_{s-i-1}v_{s-i}\}$ and adding edges
$\{u_{k-(i-1)}v_{s-i-1},u_{k-i}v_{s-i}\}$ (see Fig.\,\ref{fig:G2}). Clearly,
$G_2\cong G(u;k,s)$. In light of \eqref{eq:definition-p-spectral-radius}, we have
\begin{align*}
  0 & \geq\lambda^{(p)}(G(u;k,s))-\lambda^{(p)}(G(u;k+1,s-1)) \\
    & =\lambda^{(p)}(G_2)-\lambda^{(p)}(G(u;k+1,s-1))\\
    & \geq 2\sum_{vw\in E(G(u;k,s))} x_vx_w-2\sum_{vw\in E(G(u;k+1,s-1))} x_vx_w\\
    & =2(x_{u_{k-(i-1)}}-x_{v_{s-i}})(x_{v_{s-i-1}}-x_{u_{k-i}}).
\end{align*}
By the inductive hypothesis, we obtain $x_{u_{k-i}}\geq x_{v_{s-i-1}}$. If $x_{u_{k-i}}=x_{v_{s-i-1}}$,
then $\lambda^{(p)}(G(u;k,s))=\lambda^{(p)}(G(u;k+1,s-1))$. Using the eigenequations \eqref{eq:eigenequation}
for $u_{k-i}$, we have
\begin{align*}
0 & =(\lambda^{(p)}(G(u;k,s))-\lambda^{(p)}(G(u;k+1,s-1)))x_{u_{k-i}}^{p-1} \\
  & =x_{v_{s-i}}-x_{u_{k-(i-1)}} \\
  & <0,
\end{align*}
a contradiction. Therefore, $x_{u_{k-i}}>x_{v_{s-i-1}}$.

$\bullet$~~Therefore we obtain $x_{u_{k-i}}>x_{v_{s-i-1}}$ for $i=0,1,\ldots,s-1$.
In particular, $x_{u_{k-(s-1)}}>x_{v_0}=x_u$.

Finally, we let $G_3$ denote the graph obtained from $G(u;k+1,s-1)$ by moving edges
$\{uv: v\in N_G(u)\}$ from $u$ to $u_{k-(s-1)}$, i.e.,
\[
G_3:=G(u;k+1,s-1)-\{uv: v\in N_G(u)\}+\{u_{k-(s-1)}v: v\in N_G(u)\}.
\]
Clearly, $G_3\cong G(u;k,s)$. By \autoref{thm:move-edges},
$\lambda^{(p)}(G(u;k,s))=\lambda^{(p)}(G_3)>\lambda^{(p)}(G(u;k+1,s-1))$,
a contradiction to our assumption. This completes the proof of the theorem.
\end{proof}

%
%
%
%

\section{Proofs of \autoref{thm:main-result-path} and \autoref{thm:main-result-cycle}}
\label{sec3}

In this section, we prove \autoref{thm:main-result-path} and \autoref{thm:main-result-cycle}.
Let $\mathcal{M}^r(G)$ be the set of $r$-uniform hypergraphs with maximum $p$-spectral radius
among Berge-$G$ hypergraphs. We require the following lemma.

\begin{lemma}\label{lem:v(H)-leq-v(G)+r-2}
Let $G$ be a connected graph. If $p\geq 1$, then there exists an $r$-graph $H\in\mathcal{M}^r(G)$
such that $v(H)\leq v(G)+r-2$. Furthermore, if $p>r-1$, then $v(H)\leq v(G)+r-2$ for each $H\in\mathcal{M}^r(G)$.
\end{lemma}

\begin{proof}
Assume $p\geq 1$ and $\overline{H}\in\mathcal{M}^r(G)$. If $v(\overline{H})\leq v(G)+r-2$,
there is nothing to do. If $v(\overline{H})>v(G)+r-2$, we will construct a hypergraph $H$
from $\overline{H}$ such that $v(H)=v(G)+r-2$ and $\lambda^{(p)}(H)\geq\lambda^{(p)}(\overline{H})$.

Let $\bm{x}\in\mathbb{S}_{p,+}^{v(H)-1}$ be an eigenvector of $\overline{H}$ corresponding
to $\lambda^{(p)}(\overline{H})$. Choose $(r-2)$ vertices $v_1$, $v_2$,\,$\ldots$\,, $v_{r-2}$
from $V(\overline{H})\backslash V(G)$ such that $x_{v_1}$, $x_{v_2}$,\,$\ldots$\,,
$x_{v_{r-2}}$ are the first $(r-2)$ largest entries in $\bm{x}|_{V(\overline{H})\backslash V(G)}$.
Without loss of generality, we assume that $x_{v_1}\geq x_{v_2}\geq\cdots\geq x_{v_{r-2}}$.

Let $S:=\{e\in E(\overline{H}): e\cap (V(\overline{H})\backslash V(G))\neq\emptyset\}$. We now
define an $r$-uniform hypergraph $H$ by
\begin{equation}\label{eq:H-from-overline-H}
\begin{split}
V(H) & =V(G)\cup \{v_1,v_2,\ldots,v_{r-2}\},\\
E(H) & =(E(\overline{H})\backslash S)\cup \{(e\cap V(G))\cup\{v_1,\ldots,v_{r-|e\cap V(G)|}\}: e\in S\}.
\end{split}
\end{equation}
We also define a vector $\bm{y}\in\mathbb{S}_{p,++}^{v(H)-1}$ for $H$ as follows:
\[
y_v=\begin{cases}
      x_v, & \text{if}\ v\in V(G), \\
      \left(x_v^p+\frac{\sum_{u\in V(\overline{H})\backslash V(H)} x_u^p}{r-2}\right)^{1/p}, & \text{otherwise}.
    \end{cases}
\]
It is clear that $H$ is a Berge-$G$ hypergraph and
\[
  \lambda^{(p)}(H)-\lambda^{(p)}(\overline{H})
  \geq r\sum_{e\in E(H)} \prod_{u\in e} y_u-r\sum_{e\in E(\overline{H})} \prod_{u\in e} x_u\geq 0,
\]
which yields that $\lambda^{(p)}(H)\geq\lambda^{(p)}(\overline{H})$.

In what follows we let $p>r-1$ and prove that $v(H)\leq v(G)+r-2$ for each $H\in\mathcal{M}^r(G)$.
Suppose for a contradiction that there is an $r$-uniform hypergraph $\overline{H}\in\mathcal{M}^r(G)$ such
that $v(H)>v(G)+r-2$. Then the Berge-$G$ hypergraph $H$ given by \eqref{eq:H-from-overline-H} has
larger $p$-spectral radius than $\overline{H}$ due to $\bm{x}\in\mathbb{S}_{p,++}^{v(H)-1}$. This
is a contradiction to the fact that $\overline{H}\in\mathcal{M}^r(G)$.
\end{proof}

In \cite{Nikiforov2014:analytic-methods}, Nikiforov proved the following relation of the
$p$-spectral radius between $H$ and $H\ast K_1$. A generalized result for large $p$ can
be found in \cite{LiuLu2018:alpha-normal-method}.

\begin{lemma}[\cite{Nikiforov2014:analytic-methods}]
\label{lem:H-vee-K1}
Let $p\geq 1$ and $H$ be an $r$-uniform hypergraph. Then
\[
\lambda^{(p)}(H\ast K_1)=\frac{(r+1)^{1-(r+1)/p}}{r^{1-r/p}} \lambda^{(p)}(H).
\]
Furthermore, if $\bm{x}\in\mathbb{S}_{p,+}^{v(H)}$ is an eigenvector corresponding to $\lambda^{(p)}(H\ast K_1)$,
then $(1+1/r)^{1/p}\bm{x}|_{V(H)}$ is an eigenvector corresponding to $\lambda^{(p)}(H)$.
\end{lemma}

For a graph, the {\em diameter} of $G$, denoted by diam$(G)$, is the maximum distance between
any pair of vertices of $G$.  If $G$ has at least one cycle, the minimum length of a cycle in $G$ is
the {\em girth} $g(G)$ of $G$. The set of neighbours of a vertex $v$ in a graph $G$ is denoted by $N_G(v)$.

\begin{theorem}\label{thm:k-vertices-general-G}
Suppose that $p\geq 1$ and $G$ is a connected graph on $k$ vertices. Then
there is a $3$-graph $H\in\mathcal{M}^3(G)$ such that $v(H)=k$ if $G$ is one of the following
\begin{enumerate}
  \item[$(1)$] $G$ is a tree with diameter diam$(G)\geq 5$;
  \item[$(2)$] The girth $g(G)\geq 6$.
\end{enumerate}
Furthermore, if $p>2$, then $v(H)=k$ for each $H\in\mathcal{M}^3(G)$ when $G$
satisfies one of the above items.
\end{theorem}

\begin{proof}
Let $p\geq 1$. By \autoref{lem:v(H)-leq-v(G)+r-2}, there is a $3$-graph $\overline{H}\in\mathcal{M}^3(G)$
such that $v(\overline{H})\leq k+1$. If $v(\overline{H})=k$, we are done. Now we assume $v(\overline{H})=k+1$.
In the following we construct a $3$-graph $H$ from $\overline{H}$ such that $v(H)=k$ while
$\lambda^{(p)}(H)\geq\lambda^{(p)}(\overline{H})$. Assume that $\bm{x}\in\mathbb{S}_{p,+}^k$ is an eigenvector
of $\overline{H}$ corresponding to $\lambda^{(p)}(\overline{H})$. We shall further assume that $w$ is a vertex
in $V(G)$ such that $x_w=\max\{x_v: v\in V(G)\}$. Let $u$ denote the unique vertex in $V(\overline{H})\backslash V(G)$.

\begin{claim}\label{claim:u<w}
$x_u<x_w$.
\end{claim}

\noindent\textit{Proof of \autoref{claim:u<w}}. If $x_u\geq x_w$, then
$\lambda^{(p)}(G\ast u)\geq\lambda^{(p)}(\overline{H})$ by \autoref{thm:move-edges}.
If $\lambda^{(p)}(G\ast u)>\lambda^{(p)}(\overline{H})$, we are done by contradiction.
Otherwise, we assume $\lambda^{(p)}(G\ast u)=\lambda^{(p)}(\overline{H})$. Recall
that $G$ is either a tree with diam$(G)\geq 5$ or a connected graph with $g(G)\geq 6$,
we have two vertices $v_1$ and $v_2$ in $V(G)$ such that the distance dist$_G(v_1,v_2)=3$. Let $G'$
be the graph obtained from $G$ by deleting $v_2$ and adding edges $\{v_1v: v\in N_G(v_2)\}$.
Consider the hypergraph $G'\ast v_2$. For each edge $e\in E(G)$, if $v_2\notin e$, we embed
$e$ into $e\cup\{v_2\}$; if $v_2\in e$, we embed $e$ into $e\cup\{v_1\}$. Therefore,
$G'\ast v_2$ is a Berge-$G$ hypergraph.

If $p>1$, then $\lambda^{(p)}(G')>\lambda^{(p)}(G)$ by \autoref{thm:delete-vetex-increase-spectral-radius}.
If $p=1$, noting that the clique number of $G'$ is $\omega(G')=3$ and the clique number
of $G$ is $\omega(G)=2$. In light of Motzkin--Straus theorem (see \cite{MotzkinStraus1965}),
we have
\[
\lambda^{(1)}(G')=1-\frac{1}{3}>1-\frac{1}{2}=\lambda^{(1)}(G).
\]
It follows from \autoref{lem:H-vee-K1} that
$\lambda^{(p)}(G'\ast w)>\lambda^{(p)}(G\ast u)=\lambda^{(p)}(\overline{H})$ for $p\geq 1$.
This is a contradiction to the fact that $\overline{H}\in\mathcal{M}^3(G)$.
The proof of the claim is completed.
\vspace{2mm}

Fix $w$ and denote $D_i$ the set of vertices of $\overline{H}$ whose distance from $w$ is
$i$ in $\overline{H}$. Since $G$ either is a tree with diam$(G)\geq 5$
or a connected graph with $g(G)\geq 6$, we have $D_3\neq\emptyset$.
Let $w^*$ be a vertex in $D_3$, then $w^*v\notin E(G)$ for any $v\in N_G(w)$.

Let $G''$ be the graph obtained from $G$ by deleting $w$ and adding edges $\{w^*v: v\in N_G(w)\}$.
Now, we denote $H:=G''\ast w$. Since $x_u<x_w$, in view of \autoref{thm:move-edges} and
\autoref{thm:delete-vetex-increase-spectral-radius}, we see
\begin{equation}\label{eq:xu<xw-H>overline-H}
\lambda^{(p)}(H)\geq\lambda^{(p)}(\overline{H}),
\end{equation}
as desired. When $p>2$, then \eqref{eq:xu<xw-H>overline-H} is strict,
a contradiction to $\overline{H}\in\mathcal{M}^3(G)$.
\end{proof}

We immediately obtain the following corollary.

\begin{corollary}\label{coro:v(H)=v(G)}
Let $G$ be either a path $P_k$ or a cycle $C_k$ $(k\geq 6)$.
If $p\geq 1$, then there is a $3$-graph $H\in\mathcal{M}^3(G)$ such that $v(H)=v(G)$.
Furthermore, if $p>2$, then $v(H)=v(G)$ for each $H\in\mathcal{M}^3(G)$.
\end{corollary}

Before we show the proof of \autoref{thm:main-result-path}, we will prove the following
theorem.

\begin{theorem}\label{thm:intersect-a-vertex}
Suppose that $G$ is either a path $P_k$ or a cycle $C_k$ $(k\geq 6)$. Let $p\geq 1$, then
there is a $3$-graph $H\in\mathcal{M}^3(G)$ such that all edges of $H$ intersect at a common
vertex. Furthermore, if $p>2$, then all edges of $H$ intersect at a common vertex for each
$H\in\mathcal{M}^3(G)$.
\end{theorem}

\begin{proof}
Let $P_k:= v_1v_2\cdots v_{k-1}v_k$ be a path on $k$ vertices. By \autoref{coro:v(H)=v(G)},
there is a $\overline{H}\in\mathcal{M}^3(P_k)$ such that $v(\overline{H})=k$. Let
$\bm{x}\in\mathbb{S}_{p,+}^{k-1}$ be an eigenvector corresponding to $\lambda^{(p)}(\overline{H})$,
and $v_i$ be a vertex attaining the maximum in $\bm{x}$. By \autoref{thm:move-edges}, we
can further assume that $e$ is embedded into $e\cup\{v_i\}$ of $\overline{H}$ for each
$e\in E(P_k)\backslash\{v_{i-2}v_{i-1},v_{i-1}v_i,v_iv_{i+1},v_{i+1}v_{i+2}\}$,
and $\{v_{i-2},v_{i-1},v_i\}\in E(\overline{H})$, $\{v_i,v_{i+1},v_{i+2}\}\in E(\overline{H})$.

In the following we will construct a hypergraph $H\in\mathcal{M}^3(P_k)$ from $\overline{H}$ such that
all edges of $H$ contain vertex $v_i$. Without loss of generality, we assume $i\geq 3$.
If $v_{i-2}v_{i-1}$ and $v_{i+1}v_{i+2}$ are embedded into $\{v_{i-2}, v_{i-1},v_i\}$ and
$\{v_i,v_{i+1},v_{i+2}\}$, we are done. Otherwise, without loss of generality that $v_{i-2}v_{i-1}$
is embedded into $\{v_{i-2},v_{i-1},u\}$, where $u\neq v_i$. Then $v_{i-1}v_i$ is embedded into
$\{v_{i-2},v_{i-1},v_i\}\in E(\overline{H})$. We consider the following two cases.

{\bfseries Case 1.} $u=v_{i+1}$.

{\bfseries Subcase 1.1.} If $\{v_{i-1},v_i,v_{i+1}\}\in E(\overline{H})$, then $v_iv_{i+1}$ is embedded
into $\{v_{i-1},v_i,v_{i+1}\}$. Now we embed $v_{i-2}v_{i-1}$ into $\{v_{i-2},v_{i-1},v_i\}$, embed
$v_{i-1}v_i$ into $\{v_{i-1},v_i,v_{i+1}\}$, embed $v_iv_{i+1}$ into $\{v_{i-2},v_i,v_{i+1}\}$,
and keep the other edge embeddings as $\overline{H}$. Denote the resulting hypergraph by $H'$. Clearly,
$H'$ is a Berge-$P_k$ hypergraph. However, $\lambda^{(p)}(H')\geq\lambda^{(p)}(\overline{H})$ by
\autoref{thm:move-edges}.

{\bfseries Subcase 1.2.} If $\{v_{i-1},v_i,v_{i+1}\}\notin E(\overline{H})$, then we embed $v_{i-2}v_{i-1}$
into $\{v_{i-2},v_{i-1},v_i\}$, embed $v_{i-1}v_i$ into $\{v_{i-1},v_i,v_{i+1}\}$, and keep the
other edge embeddings as $\overline{H}$. Denote the resulting hypergraph by $H'$. Then
$\lambda^{(p)}(H')\geq\lambda^{(p)}(H)$ by \autoref{thm:move-edges}.

{\bfseries Case 2.} $u\neq v_{i+1}$. We embed $v_{i-2}v_{i-1}$ into $\{v_{i-2}, v_{i-1},v_i\}$,
embed $v_{i-1}v_i$ into $\{v_{i-1},v_i,u\}$, and keep the other edge embeddings as $\overline{H}$.
Denote the resulting hypergraph by $H'$. Then $\lambda^{(p)}(H')\geq\lambda^{(p)}(H)$ by
\autoref{thm:move-edges}.

Now, if $v_iv_{i+1}$ is embedded into $\{v_i,v_{i+1},v_{i+2}\}$, we choose $H=H'$ and obtain the
desired $H$. If not, we can repeat the above discussion for $v_{i+1}v_{i+2}$ similarly, and get
the desired $H$.

Finally, if $p>2$, then the above inequalities are both strict, which yields a contradiction.
The result follows.

The proof for cycle is analogous to that for path, so we omit the proof.
\end{proof}

Let $F(u,v;h,j)$ be the graph obtained from a cycle $C_{\ell}$ by attaching two pendent paths
of lengths $h\geq 0$, $j\geq 0$ at vertices $u\in V(C_{\ell})$ and $v\in V(C_{\ell})$, respectively.
For graphs, $G$ is called {\em $F(u,v;h,j)$-type} if $G\cong F(u,v;h,j)$ for some $u=u^*$, $v=v^*$,
$h=h^*$ and $j=j^*$.

We are now ready to prove \autoref{thm:main-result-path} and \autoref{thm:main-result-cycle}.
\vspace{3mm}

\noindent{\bfseries Proof of \autoref{thm:main-result-path}.}~~
Let $p\geq 1 $ and $P_k:= v_1v_2\cdots v_{k-1}v_k$ be a path on $k$ vertices.
Assume that $H\in\mathcal{M}^3(P_k)$, $\bm{x}\in\mathbb{S}_{p,+}^{k-1}$ is an eigenvector of $H$
corresponding to $\lambda^{(p)}(H)$, and $v_i$ is a vertex such that $x_{v_i}=\max\{x_v: v\in V(H)\}$.
By \autoref{thm:intersect-a-vertex}, we can further assume that $v(H)=k$ and all edges of $H$
intersect at the common vertex $v_i$.

In other words, each edge $e\in E(P_k)\backslash \{v_{i-1}v_i,v_iv_{i+1}\}$ is embedded into
$e\cup\{v_i\}$ of $H$. Suppose that $v_{i-1}v_i$ is embedded into $\{v_{i-1},v_i,v_s\}\in E(H)$
and $v_iv_{i+1}$ is embedded into $\{v_i,v_{i+1},v_t\}$. Let $G^*$ be the core of $H$, that is,
$H=G^*\ast v_i$.

Clearly, the clique number of $G^*$ at most $3$. If $p=1$, then $\lambda^{(1)}(G^*)\leq 2/3$.
The result follows from \autoref{lem:H-vee-K1}. In the following we assume $p>1$.

{\bfseries Case 1.} $i\geq k-2$. In this case the core $G^*$ of $H$ is $F(u,v;h,j)$-type.

{\bfseries Case 2.} $i\leq k-3$. If $s\leq i-3$ and $t\geq i+3$, it is claimed that we can
further assume that $v_s=v_t$. If not, we let $H_1$ be the hypergraph obtained from $H$ by
moving edge $\{v_{i-1},v_i,v_s\}$ from $v_s$ to $v_t$, and let $H_2$ be the hypergraph obtained
from $H$ by moving edge $\{v_i,v_{i+1},v_t\}$
from $v_t$ to $v_s$. Obviously, both $H_1$ and $H_2$ are Berge-$P_k$ hypergraphs. Since either
$x_{v_s}\geq x_{v_t}$ or $x_{v_s}\leq x_{v_t}$, in view of \autoref{thm:move-edges}, we obtain
\[
\max\{\lambda^{(p)}(H_1),\lambda^{(p)}(H_2)\}\geq\lambda^{(p)}(H).
\]
Therefore, $G^*$ is $F(u,v;h,j)$-type.
If $s\geq i-2$ or $t\leq i+2$, we also have that $G^*$ is $F(u,v;h,j)$-type.

According to Case 1, Case 2 and \autoref{lem:H-vee-K1}, it suffices to determine the graphs with
maximum $p$-spectral radius among $F(u,v;h,j)$ on $(k-1)$ vertices. Using a similar arguments as
Case 2, we see $u=v$.

For short, we denote $C_{\ell}(u;h,j):=F(u,u;h,j)$, the graph obtained from a $\ell$-cycle $C_{\ell}$ by
attaching two pendent paths of lengths $h$ and $j$ at $u\in V(C_{\ell})$, respectively (see Fig.\,\ref{fig:core}).

\begin{figure}[htbp]
\centering
\begin{tikzpicture}
  \foreach \i in {0,1,3,4,5,7}
  {
  \coordinate (v\i) at (45*\i:1.2);
  \filldraw (v\i) circle (0.06);
  }
  \foreach \i in {-3,-2,...,3}
  {
  \coordinate (u\i) at (1.2*\i,1.3);
  \filldraw (u\i) circle (0.06);
  }
  \node[above] at (u-1) {$u_1$};
  \node[above] at (u-2) {$u_{j-1}$};
  \node[above] at (u-3) {$u_j$};
  \node[above] at (0,1.3) {$u$};
  \node[above] at (u1) {$v_1$};
  \node[above] at (u2) {$v_{h-1}$};
  \node[above] at (u3) {$v_h$};
  \node[right] at (v0) {$w_2$};
  \node[right] at (v1) {$w_1$};
  \node[left] at (v3) {$w_{\ell-1}$};
  \node[left] at (v4) {$w_{\ell-2}$};
  \node at ($(u-2)!0.5!(u-1)$) {$\cdots$};
  \node at ($(u1)!0.5!(u2)$) {$\cdots$};
  \draw (u-3)--(u-2);
  \draw (u-1)--(0,1.3)--(u1);
  \draw (u2)--(u3);
  \draw (v0)--(v1)--(0,1.3)--(v3)--(v4)--(v5);
  \draw (v7)--(v0);
  \filldraw (255:1.1) circle (0.02);
  \filldraw (270:1.1) circle (0.02);
  \filldraw (285:1.1) circle (0.02);
\end{tikzpicture}
  \caption{\label{fig:core} Graph $C_{\ell}(u;h,j)$}
\end{figure}
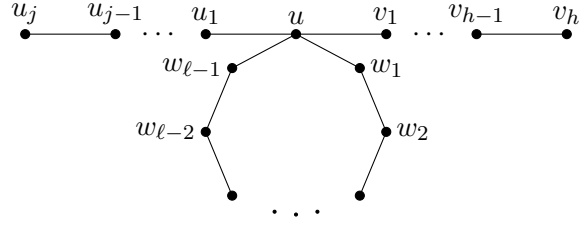

In the following we assume that $C_{g}(u;a,b)$ is a graph attaining the maximum $p$-spectral radius
among $C_{\ell}(u;h,j)$. Let $\bm{y}\in\mathbb{S}_{p,++}^{k-2}$ be an eigenvector to $\lambda^{(p)}(C_{g}(u;a,b))$.
Without loss of generality, we assume $a\geq b\geq 0$.

\begin{claim}\label{claim:g<k-1}
$g<k-1$.
\end{claim}

\noindent\textit{Proof of \autoref{claim:g<k-1}.}
If not, then $C_g(u;a,b)$ is exactly $C_{k-1}$. Clearly, $\lambda^{(p)}(C_{k-3}(u;1,1))>\lambda^{(p)}(C_{k-1})$,
a contradiction to $\lambda^{(p)}(C_{k-1}\ast K_1)\in\mathcal{M}^3(P_k)$. The proof of the claim is completed.

\begin{claim}\label{claim:g=3}
$g=3$.
\end{claim}

\noindent\textit{Proof of \autoref{claim:g=3}.}
Suppose for a contradiction that $g\geq 4$. By \autoref{claim:g<k-1}, $a\geq 1$. We first prove
that $y_{v_1}<y_{w_2}$.

{\bfseries Case 1. } $a=1$. Let $G_4$ be the graph obtained from $C_g(u;1,b)$ by deleting
edge $w_1w_2$ and adding edge $v_1w_1$. Then
\[
0\geq\lambda^{(p)}(G_4)-\lambda^{(p)}(C_g(u;1,b)) \geq
2(y_{v_1}-y_{w_2})y_{w_1},
\]
which implies that $y_{v_1}\leq y_{w_2}$. If $y_{v_1}=y_{w_2}$, then $\lambda^{(p)}(G_4)=\lambda^{(p)}(C_g(u;1,b))$,
and $\bm{y}$ is also an eigenvector of $G_4$ corresponding to $\lambda^{(p)}(G_4)$. By using eigenequations \eqref{eq:eigenequation}
for $v_1$, we have
\[
0=(\lambda^{(p)}(G_4)-\lambda^{(p)}(C_g(u;1,b))) y_{v_1}^{p-1}=y_{w_1}>0,
\]
a contradiction. Hence, $y_{v_1}<y_{w_2}$.

{\bfseries Case 2.} $a=2$. In view of eigenequations \eqref{eq:eigenequation}, we see
\[
\lambda^{(p)}(C_g(u;2,b)) y_{v_2}^{p-1}=y_{v_1},~~\lambda^{(p)}(C_g(u;2,b)) y_{w_1}^{p-1}
=y_u+y_{w_2}.
\]
It follows that $y_{v_2}<y_{w_1}$. Denote
\[
G_5:=C_g(u;2,b)-\{w_1w_2,v_1v_2\}+\{v_1w_1,v_2w_2\}.
\]
Therefore we have
\[
0\geq\lambda^{(p)}(G_5)-\lambda^{(p)}(C_g(u;2,b))\geq 2(y_{v_1}-y_{w_2})(y_{w_1}-y_{v_2}),
\]
which yields that $y_{v_1}\leq y_{w_2}$. Using the similar argument as Case 1, we have
$y_{v_1}<y_{w_2}$.

{\bfseries Case 3.} $a\geq 3$.
Let $G_5$ be the graph obtained from $C_g(u;a,b)$ by deleting edges $\{uw_1,v_2v_3\}$
and adding edges $\{uv_2,w_1v_3\}$. Then
\[
0\geq\lambda^{(p)}(G_5)-\lambda^{(p)}(C_g(u;a,b))\geq 2(y_u-y_{v_3})(y_{v_2}-y_{w_1}).
\]
Recall that $y_u>y_{v_3}$. Hence, $y_{v_2}\leq y_{w_1}$. Let $G_6$ be the graph
obtained from $C_g(u;a,b)$ by deleting edges $\{w_1w_2,v_1v_2\}$ and adding
edges $\{v_1w_1,v_2w_2\}$. Then
\[
0\geq\lambda^{(p)}(G_6)-\lambda^{(p)}(C_g(u;a,b))\geq 2(y_{v_1}-y_{w_2})(y_{w_1}-y_{v_2}).
\]
It follows from $y_{v_2}\leq y_{w_1}$ that $y_{v_1}\leq y_{w_2}$. Using the similar
argument as Case 1, we have $y_{v_1}\neq y_{w_2}$. Therefore, $y_{v_1}<y_{w_2}$.

According to Case 1--3, we have $y_{v_1}<y_{w_2}$. Now we let $G_7$ be the graph obtained
from $C_g(u;a,b)$ by contracting edges $\{uw_1,w_1w_2\}$ and subdividing edge $uv_1$. In
light of $y_u>y_{w_1}$, we have
\[
\lambda^{(p)}(G_7)-\lambda^{(p)}(C_g(u;a,b))\geq 2(y_u-y_{w_1})(y_{w_2}-y_{v_1})>0,
\]
which yields a contradiction. The proof of the claim is completed.
\par\vspace{3mm}

Finally, in view of \autoref{claim:g=3} and \autoref{thm:G(k,s)>G(k+1,s-1)},
$\lambda^{(p)}(G^*)\leq \lambda^{(p)}(\Delta_1)$. It follows that
$\lambda^{(p)}(H)\leq\lambda^{(p)}(\Delta_1\ast K_1)$ from \autoref{lem:H-vee-K1}.
Furthermore, if $p>2$, we know $H\in\mathcal{M}^3(P_k)$ is unique by \autoref{thm:intersect-a-vertex}.
\hfill\ensuremath{\Box}
\par\vspace{4mm}

\noindent{\bfseries Proof of \autoref{thm:main-result-cycle}.} Using the similar argument
as the proof of \autoref{thm:main-result-path}, we can obtain
\autoref{thm:main-result-cycle}. \hfill\ensuremath{\Box}

\section{Proof of \autoref{thm:main-result-star}}
\label{sec4}

In this section, we shall prove \autoref{thm:main-result-star}. Before continuing, we
need the following lemma.

\begin{lemma}[\cite{Nikiforov2014:analytic-methods}]
\label{lem:nonincreasing}
Let $p\geq 1$, and $H$ be an $r$-uniform hypergraph with $m$ edges. Then the function
$(\lambda^{(p)}(H)/(rm))^p$ is nonincreasing in $p$. That is, for $p<q$ we have
\begin{equation}\label{eq:nonincreasing}
\bigg(\frac{\lambda^{p}(H)}{rm}\bigg)^p\geq
\bigg(\frac{\lambda^{q}(H)}{rm}\bigg)^q.
\end{equation}
Furthermore, let $\bm{x}\in\mathbb{S}_{q,+}^{v(H)-1}$ be an eigenvector to
$\lambda^{(q)}(H)$, if equality holds in \eqref{eq:nonincreasing}, then
$\prod_{u\in e} x_u=\prod_{v\in f} x_v$ for any edges $e\neq f$ of $H$.
\end{lemma}

\begin{lemma}\label{lem:Sn>Sn^+}
Let $p>1$ and $n\geq 10$. Let $S_{n-1}^+$ be the graph obtained from star $S_{n-1}$ by
adding an edge. \\
$(1)$ If $p\geq 2$, then $\lambda^{(p)}(S_n)>\lambda^{(p)}(S_{n-1}^+)$ unless $p=2$, $n=10$.\\
$(2)$ If $1<p<2$, then $\lambda^{(p)}(S_n)>\lambda^{(p)}(S_{n-1}^+)$ for sufficiently large $n$.
\end{lemma}

\begin{proof}
(1). Assume $p=2$. It is clear that $\lambda^{(2)}(S_n)=\sqrt{n-1}$, and $\lambda^{(2)}(S_{n-1}^+)$
is the largest root of the following equation
\[
x^3-x^2-(n-2)x+n-4=0.
\]
After some algebra we have $\lambda^{(2)}(S_{n-1}^+)<\sqrt{n-1}=\lambda^{(2)}(S_n)$ for
$n\geq 11$ and $\lambda^{(2)}(S_{n-1}^+)=\lambda^{(2)}(S_n)$ for $n=10$.

Now we assume $p>2$. In light of \autoref{lem:nonincreasing} we see
\[
\bigg(\frac{\lambda^{(p)}(S_{n-1}^+)}{2(n-1)}\bigg)^p<
\bigg(\frac{\lambda^{(2)}(S_{n-1}^+)}{2(n-1)}\bigg)^2\leq
\frac{1}{4(n-1)},
\]
which yields that $\lambda^{(p)}(S_{n-1}^+)<2^{1-2/p}(n-1)^{1-1/p}$. Noting that
$\lambda^{(p)}(S_n)=2^{1-2/p}(n-1)^{1-1/p}$ (see also \cite[Proposition 4.9]{Nikiforov2014:analytic-methods}),
we get $\lambda^{(p)}(S_n)>\lambda^{(p)}(S_{n-1}^+)$.

(2). Let $1<p<2$.
Suppose that $v_1$ is the vertex of $S_{n-1}^+$ with degree $n-2$, and $v_2v_3$,
$v_1v_i$ are the edges of $S_{n-1}^+$, $i=2$, $3$, $\ldots$, $n-1$. Let
$\bm{x}\in\mathbb{S}_{p,++}^{n-2}$ be the eigenvector corresponding to $\lambda^{(p)}(S_{n-1}^+)$.
By symmetry,
\[
x_{v_2}=x_{v_3},~x_{v_4}=x_{v_5}=\cdots=x_{v_{n-1}}.
\]
For short, set $x:=x_{v_1}$, $y:=x_{v_4}$, $z:=x_{v_2}$ and $\lambda:=\lambda^{(p)}(S_{n-1}^+)$.
By eigenequations \eqref{eq:eigenequation}, we see
\begin{equation}\label{eq:eigenequation-lambda}
\begin{cases}
  \lambda x^{p-1}=2z+(n-4)y, \\
  \lambda y^{p-1}=x, \\
  \lambda z^{p-1}=x+z.
\end{cases}
\end{equation}

In what follows, we tacitly assume that $n$ is large enough. Since
$\lambda>\lambda^{(p)}(S_{n-1})=2^{1-2/p}(n-2)^{1-1/p}$ and $\lambda z^{p-1}=x+z$,
we have
\[
z<\Big(\frac{2}{\lambda}\Big)^{1/(p-1)}<\frac{4^{1/(p^2-p)}}{(n-2)^{1/p}},
\]
which implies that $z=O(n^{-1/p})$. Setting $\varepsilon:=z^{2-p}/\lambda$, then
$\varepsilon=O(n^{-1/p})$.

Solving system \eqref{eq:eigenequation-lambda}, we find that
\begin{equation}\label{eq:lambda^p}
\begin{split}
\lambda^p
 & =\frac{(n-4)(1-z^{2-p}/\lambda)^{1/(p-1)}+2}{(1-z^{2-p}/\lambda)^{p-1}}\cdot z^{p(2-p)} \\
 & =\frac{(n-4)(1-\varepsilon)^{1/(p-1)}+2}{(1-\varepsilon)^{p-1}}\cdot z^{p(2-p)}.
\end{split}
\end{equation}
In view of $||\bm{x}||_p=x^p+(n-4)y^p+2z^p=1$ and \eqref{eq:eigenequation-lambda}, we have
\[
\lambda^pz^{p(p-1)}(1-\varepsilon)^p+
(n-4)z^p(1-\varepsilon)^{p/(p-1)}+2z^p=1.
\]
Combining \eqref{eq:lambda^p} we obtain
$2z^p[(n-4)(1-\varepsilon)^{p/(p-1)}-\varepsilon+2]=1$.
Hence,
\begin{align*}
z^p & =\frac{1}{2[(n-4)(1-\varepsilon)^{p/(p-1)}-\varepsilon+2]} \\
    & =\frac{1+p/(p-1)\varepsilon+O(\varepsilon^2)}{2(n-2)}.
\end{align*}
By some algebra, we immediately have
\begin{equation}\label{eq:z^(2p-p2)}
z^{p(2-p)}=\frac{1+(2p-p^2)/(p-1)\varepsilon+O(\varepsilon^2)}{2^{2-p}(n-2)^{2-p}}.
\end{equation}

In the following we shall give an estimation of $\lambda^p/(\lambda^{(p)}(S_n))^p$ by \eqref{eq:lambda^p} and \eqref{eq:z^(2p-p2)}. Recall that $\lambda^{(p)}(S_n)=2^{1-2/p}(n-1)^{1-1/p}$,
we deduce that
\begin{align*}
\frac{\lambda^p}{(\lambda^{(p)}(S_n))^p} & =
\frac{(n-4)(1-\varepsilon)^{1/(p-1)}+2}{2^{p-2}(n-1)^{p-1}(1-\varepsilon)^{p-1}}\cdot z^{p(2-p)} \\
& =\frac{(n-2)[1-(2p-p^2)/(p-1)\varepsilon+O(\varepsilon^2)]}{2^{p-2}(n-1)^{p-1}}\cdot z^{p(2-p)} \\
& =(n-2)\Big[1-\frac{2p-p^2}{p-1}\varepsilon+O(\varepsilon^2)\Big]\cdot z^{p(2-p)} \\
& =\left(1-\frac{1}{n-1}\right)^{p-1}(1+O(\varepsilon^2))\\
& =1-\frac{p-1}{n-1}+O(n^{-2/p})\\
& <1,
\end{align*}
which yields that $\lambda<\lambda^{(p)}(S_n)$. The proof of the lemma is completed.
\end{proof}

We are now ready to prove \autoref{thm:main-result-star}. \par\vspace{2mm}

\noindent{\bfseries Proof of \autoref{thm:main-result-star}.}
Assume $p\geq 1$ and $c$ is the center of $S_k$. Clearly, all edges of Berge-$S_k$ hypergraph
contain vertex $c$. If $p=1$, there is a hypergraph $G^*\ast c\in\mathcal{M}^3(S_k)$ such that
$v(G^*\ast c)\leq k+1$ by \autoref{lem:v(H)-leq-v(G)+r-2}. Hence, $G^*$ is either a tree or a
unicyclic graph. It follows from Motzkin--Straus theorem and \autoref{lem:H-vee-K1} that
$\lambda^{(1)}(H)\leq\lambda^{(1)}(S_{k-1}^+\ast K_1)$ for any Berge-$S_k$ hypergraph $H$. In the
following we assume $p>1$ and $H=G^*\ast c\in\mathcal{M}^3(S_k)$.

\begin{claim}\label{claim:v(H)=k-k+1}
$v(H)\in\{k,k+1\}$.
\end{claim}

\noindent\textit{Proof of \autoref{claim:v(H)=k-k+1}.}
Let $\bm{x}\in\mathbb{S}_{p,+}^{v(H)-1}$ be an eigenvector corresponding to $\lambda^{(p)}(H)=\lambda^{(p)}(G^*\ast c)$.
By \autoref{lem:H-vee-K1}, $3/2\,\bm{x}|_{V(G^*)}$ is an eigenvector corresponding to $\lambda^{(p)}(G^*)$.
In view of \autoref{thm:positive-vector-p>r-1}, $\bm{x}|_{V(G^*)}$ is a positive vector, and then
$\bm{x}\in\mathbb{S}_{p,++}^{v(H)-1}$. Using the same argument as \autoref{lem:v(H)-leq-v(G)+r-2},
we see $v(H)\in\{k,k+1\}$, completing the proof of the claim. \par\vspace{2mm}

Now, we consider the following two cases.

{\bfseries Case 1.} $v(H)=k$.
Then $G^*$ is a unicyclic graph. Hence, $\lambda^{(p)}(G^*)< \lambda^{(p)}(S_{k-1}^+)$ unless $G^*=S_{k-1}^+$.

{\bfseries Case 2.} $v(H)=k+1$. Then $G^*$ is a tree. It follows that
$\lambda^{(p)}(G^*)<\lambda^{(p)}(S_k)$ unless $G^*=S_k$.

According to Case 1 and Case 2, we have
\[
\lambda^{(p)}(G^*)\leq\max\{\lambda^{(p)}(S_{k-1}^+),~\lambda^{(p)}(S_k)\}.
\]
The result follows from \autoref{lem:Sn>Sn^+} and \autoref{lem:H-vee-K1}.
\hfill\ensuremath{\Box}

\begin{remark}
For a graph $G$ and let $r\geq 3$ be an integer. The $r$-uniform {\em expansion} of $G$ is the
$r$-uniform hypergraph obtained from $G$ by enlarging each edge of $G$ with $(r-2)$ new vertices
disjoint from $V(G)$ such that distinct edges of $G$ are enlarged by distinct vertices. By the
same argument as \cite[Theorem 3.8]{LiuLu2018:alpha-normal-method}, the $r$-uniform expansion
of $G$ attains the minimum $p$-spectral radius among Berge-$G$ hypergraphs $\mathcal{B}_r(G)$
for $p>r-1$.
\end{remark}

\section*{Acknowledgements}

The research of the first author was supported in part by the National Nature Science Foundation
of China (No. 11471210). The second author thanks the support of the fund from the China Scholarship
Council (No. 201706890045) when this author visited the University of South Carolina during September
2017\,--\,March 2019. The research of the third author was supported in part by NSF grant DMS-1600811
and ONR grant N00014-17-1-2842.

\end{document}